\documentclass[runningheads]{llncs}
\usepackage{graphicx}
\usepackage{amsmath}
\usepackage{amsfonts}
\usepackage{amssymb}
\usepackage{algorithm}

\usepackage{setspace}
\usepackage[misc,geometry]{ifsym}
\usepackage{cite}

\usepackage{hyperref}
\usepackage{cleveref}
\usepackage{algpseudocode}

\newcommand\yes{\textsc{Yes}}
\newcommand\no{\textsc{No}}
\newcommand\np{\textsc{NP}}
\newcommand{\set}[1]{\ensuremath{ \left\lbrace #1 \right\rbrace }}

\newcommand{\wps}{\textsc{Weighted Positive \mbox{2-SAT}}}
\newcommand\contracindep{\textsc{Con\-trac\-tion Blocker($\alpha$)}}
\newcommand\dcontracindep{\textsc{$d$-Con\-trac\-tion Blocker($\alpha$)}}
\newcommand\contracclique{\textsc{1-Con\-trac\-tion Blocker($\omega$)}}
\newcommand\delindep{\textsc{1-Dele\-tion Blocker($\alpha$)}}

\newcommand{\mvc}{\textsc{Vertex Cover}}
\newcommand{\mis}{\textsc{Inde\-pendent Set}}
\newcommand{\dist}{\textrm{dist}}

\usepackage{tikz}
\usetikzlibrary{arrows,automata}
\usetikzlibrary{decorations.pathreplacing}
\usepackage{tikz-layers}
\tikzstyle{vertex}=[thick,circle,inner sep=0.cm, minimum size=2mm, fill=white, draw=black]
\tikzstyle{hedge}=[thick]
 \tikzstyle{gedge}=[thick, draw = ForestGreen, opacity = 1]
 \tikzstyle{bcircle}=[thick, circle, minimum size = 15pt, draw]

\newcommand{\gadget}[2]{
	\def\radi{#2}
	\node[] at (0,0){$K_{#1}$};
	\draw[thick] (0,0) circle(#2);
	\tangent{0}{0}{\radi}{0}{2.4*\radi}
	\node[vertex, label = left:$v_{#1}$] (v) at (0, 2.4*\radi){};
}
\newcommand{\tangent}[5]{%circle x, circle y, radius, point x, point y
	\def\r{#3};
	\def\cx{#1};
	\def\cy{#2};
	\pgfmathsetmacro\dx{#4-#1};
	\pgfmathsetmacro\dy{#5-#2};
	
	\pgfmathsetmacro{\lend}{\dx*\dx+\dy*\dy}
	\pgfmathsetmacro{\tmp}{{sqrt(\lend-\r*\r)}}
;
	
	\pgfmathsetmacro{\xa}{\cx + (\dx*\r-\dy*\tmp)*\r/\lend}
	\pgfmathsetmacro{\ya}{\cy + (\dy * \r + (\dx * \tmp))* \r / \lend}
	\pgfmathsetmacro\xb{\cx+ (\dx*\r+\dy*\tmp)*\r/\lend};
	\pgfmathsetmacro\yb{\cy+ (\dy*\r-\dx*\tmp)*\r/\lend};
	\draw[hedge](\xa,\ya) -- (#4,#5);
	\draw[hedge](\xb,\yb) -- (#4,#5);
	
}

\begin{document}
\title{Using Edge Contractions and Vertex Deletions to Reduce the Independence Number and the Clique Number}
\author{Felicia Lucke\orcidID{0000-0002-9860-2928} \and
Felix Mann({\scriptsize\Letter})\orcidID{0000-0003-0016-4024}}
\authorrunning{F. Lucke and F. Mann}
\institute{Université de Fribourg, Boulevard de Pérolles 90, 1700 Fribourg, Switzerland
\email{\{felicia.lucke,felix.mann\}@unifr.ch}}
\maketitle              % typeset the header of the contribution
\begin{abstract}
We consider the following problem: for a given graph $G$ and two integers $k$ and $d$, can we apply a fixed graph operation at most $k$ times in order to reduce a given graph parameter $\pi$ by at least $d$?
We show that this problem is NP-hard when the parameter is the independence number and the graph operation is vertex deletion or edge contraction, even for fixed $d=1$ and when restricted to chordal graphs. 
We also give a polynomial time algorithm for bipartite graphs when the operation is edge contraction, the parameter is the independence number and $d$ is fixed.
Further, we complete the complexity dichotomy on $H$-free graphs when the parameter is the clique number and the operation is edge contraction by showing that this problem is NP-hard in $(C_3+P_1)$-free graphs even for fixed $d=1$. 
Our results answer several open questions stated in [Diner et al., Theoretical Computer Science, 746, p. 49-72 (2012)].

\end{abstract}

\section{Introduction}

Blocker problems are a type of graph modification problems which are characterised by a set $\mathcal{O}$ of graph modification operations (for example vertex deletion or edge contraction), a graph parameter $\pi$ and an integer threshold $d\geq 1$. 
The aim of the problem is to determine, for a given graph $G$, the smallest sequence of operations from $\mathcal{O}$ which transforms $G$ into a graph $G'$ such that $\pi(G')\leq \pi(G)-d$.

%Blocker problems are interesting from both a structural and a complexity point of view.
%Besides determining the computational complexity of these problems in different graph classes, studying the structure of the sets whose modification changes a given graph parameter yields knowledge about the ``critical'' parts of a graph. 

As in the case of regular graph modification problems, we often consider a set of operations consisting of a single graph operation, typically vertex deletion, edge contraction, edge addition or edge deletion. Amongst the parameters which have been studied are the chromatic number $\chi$ (see \cite{Chromatic}), the matching number $\mu$ (see \cite{Matching}), the length of a longest path (see \cite{Paths1,Paths2}), the (total or semitotal) domination  number $\gamma$ ($\gamma_t$ and $\gamma_{t2}$, respectively) (see \cite{TotalDom,domination,SemiTotalDom}), the clique number~$\omega$~(see \cite{clique}) and the independence number $\alpha$ (see \cite{Stable}).

In this paper, the set of allowed graph operations will always consist of only one operation, either \emph{vertex deletion} or \emph{edge contraction}.
Given a graph $G$, we denote by $G-U$ the graph from which a subset of vertices $U\subseteq V(G)$ has been deleted. 
Given an edge $uv\in E(G)$, contracting the edge $uv$ means deleting the vertices $u$ and $v$ and replacing them with a single new vertex which is adjacent to every neighbour of $u$ or $v$.
We denote by $G/S$ the graph in which every edge from an edge set $S\subseteq E(G)$ has been contracted.
We consider the following two problems, where $d\geq 1$ is a fixed integer.

\begin{center}
\fbox{
\begin{minipage}{4.5in}
\textsc{$d$-Deletion Blocker ($\pi$)}
\begin{description}
\item[Instance:] A graph $G$ and an integer $k$.
\item[Question:] Is there a set $U \subseteq V(G)$, $|U| \leq k$, such that \\$\pi(G-U) \leq \pi(G)-d$? 
\end{description}
\end{minipage}}
\end{center}
\begin{center}
\fbox{
\begin{minipage}{4.5in}
\textsc{$d$-Contraction Blocker ($\pi$)}
\begin{description}
\item[Instance:] A graph $G$ and an integer $k$.
\item[Question:] Is there a set $S \subseteq E(G)$, $|S| \leq k$, such that $\pi(G /S) \leq \pi(G)-d$? 
\end{description}
\end{minipage}}
\end{center}

When $d$ is not fixed but part of the input, the problems are called \textsc{Deletion Blocker}($\pi$) and \textsc{Contraction Blocker}($\pi$), respectively.

When $\pi=\alpha$ or $\pi=\omega$, both problems above are NP-hard on general graphs \cite{CliqueLatest}, so it is natural to ask if these problems remain NP-hard when the input is restricted to a special graph class.

\begin{table}[]
\caption{\label{ComplexTable}The table of complexities for some graph classes. Here, P means solvable in polynomial time, whereas NP-h and NP-c mean NP-hard and NP-complete, respectively. A question mark means that the case is open. Everything in \textbf{bold} are new results from this paper, all other cases are referenced in \cite{CliqueLatest}, where an older version of this table is given.}
\centering
\begin{tabular}{| llllll |}
\hline
Class       & \multicolumn{2}{l}{\textsc{Contraction Blocker}($\pi$)} &  & \multicolumn{2}{l|}{\textsc{Deletion Blocker}($\pi$)} \\ \cline{2-3} \cline{5-6} 
            & $\pi=\alpha$               & $\pi=\omega$               &  & $\pi=\alpha$              & $\pi=\omega$             \\ \hline\hline
Tree        & \textsc{P}                          & \textsc{P}                          &  & \textsc{P}                         & \textsc{P}                        \\\hline
Bipartite   & \textsc{NP}-h;                       & \textsc{P}                          &  & \textsc{P}                         & \textsc{P}                        \\
					&  \hspace{4pt}\textbf{d fixed: \textsc{P}}  &  & &   &\\\hline
Cobipartite & $d=1$: \textsc{NP}-c                  & \textsc{NP}-c;            &  & \textsc{P}                         & \textsc{P}                        \\
& & \hspace{4pt} $d$ fixed: \textsc{P} & & &\\\hline
Cograph     & \textsc{P}                          & \textsc{P}                          &  & \textsc{P}                         & \textsc{P}                        \\\hline
Split       & \textsc{NP}-c;           & \textsc{NP}-c;           &  & \textsc{NP}-c;          & \textsc{NP}-c;          \\
				& \hspace{4pt} $d$ fixed: \textsc{P} &  \hspace{4pt}$d$ fixed: \textsc{P} & &  \hspace{4pt}$d$ fixed: \textsc{P} &  \hspace{4pt}$d$ fixed: \textsc{P} \\\hline
Interval    & {?}                          & \textsc{P}                          &  & {?}                           & \textsc{P}                        \\\hline
Chordal     & \textbf{d=1: \textsc{NP}-c}                      & $d=1$: \textsc{NP}-c                  &  & \textbf{d=1: \textsc{NP}-c}                      & $d=1$: \textsc{NP}-c                \\\hline
Perfect     & $d=1$: \textsc{NP}-h                  & $d=1$: \textsc{NP}-h                  &  & \textbf{d=1: \textsc{NP}-c}                      & $d=1$: \textsc{NP}-c                \\ \hline
\end{tabular}

\end{table}

The authors of \cite{CliqueLatest} show that \textsc{Contraction Blocker}($\alpha$) in bipartite and chordal graphs as well as \textsc{Deletion Blocker}($\alpha$) in chordal graphs are \textsc{NP}-hard when the threshold $d$ is part of the input.
However, as an open question, they ask for the complexity of the problem when $d$ is fixed.
We show that \textsc{Contraction Blocker}($\alpha$) in bipartite graphs is solvable in polynomial time if $d$ is fixed and that the other problems are \np-hard even if $d=1$. 
An overview of the complexities in some graph classes is given in \Cref{ComplexTable}.

A \emph{monogenic} graph class is characterised by a single forbidden induced subgraph $H$. 
For a given graph parameter $\pi$, it is interesting to establish a \emph{complexity dichotomy for monogenic graphs}, that is, to determine the complexity of \textsc{\mbox{($d$-)Deletion} Blocker($\pi$)} or \textsc{($d$-)Contraction Blocker($\pi$)} in $H$-free graphs, for every graph $H$.
For example, such a dichotomy has been established for \textsc{Deletion Blocker}($\pi$) for all $\pi\in\set{\alpha,\omega,\chi}$ and \textsc{Contraction Blocker}($\pi$) for $\pi\in\set{\alpha,\chi}$ (all \cite{CliqueLatest}), \textsc{Contraction Blocker}($\gamma_{t2}$) (for $d=k=1$, \cite{SemiTotalDom}), \textsc{Contraction Blocker}($\gamma_{t}$) (for $d=k=1$, \cite{TotalDom}) and \textsc{Contraction Blocker}($\gamma$) (for $d=k=1$, \cite{domination}).
In \cite{CliqueLatest}, the computational complexity of \textsc{Contraction Blocker}($\omega$) in $H$-free graphs has been determined for every $H$ except $H=C_3+P_1$. 
We show that this case is \np-hard even when $d=1$ and complete hence the dichotomy.

The paper is organised as follows: 
In \Cref{preliminaries} we explain notation and terminology. 
In \Cref{hardness} we give the proofs of NP-hardness or NP-completeness of the aforementioned problems.
Finally, in \Cref{algorithms} we give a polynomial-time algorithm for \dcontracindep\ in bipartite graphs.

\section{Preliminaries}\label{preliminaries}

Throughout this paper, we assume that all graphs are connected unless stated differently. 

We refer the reader to \cite{diestel} for any terminology not defined here.

For a graph $G$ we denote by $V(G)$ the vertex set of the graph and by $E(G)$ its edge set.
For two graphs $G$ and $H$ we denote by $G+H$ the disjoint union of $G$ and $H$.
For two vertices $u,v \in V(G)$ we denote by $\dist_G(u,v)$ the \emph{distance} between $u$ and $v$, which is the number of edges in a shortest path between $u$ and $v$. 
For two sets of vertices $U,W\subseteq V(G)$, the \emph{distance between $U$ and $W$}, denoted by $\dist_G(U,W)$, is given by $\min_{u\in U, w\in W}\dist_G(u,w)$. 
For a set of edges $S\subseteq E(G)$ we denote by $V(S)$ the set of vertices in $V(G)$ which are endpoints of at least one edge of $S$.
Let $v \in V(G)$, then the \emph{closed neighbourhood of $v$}, denoted by $N_G[v]$, is the set $\set{u\in V(G):\,\dist_G(u,v) \leq 1}$. 
Similarly, we define for a set $U \subseteq V(G)$ the closed neighbourhood of $U$ as $N_G[U] = \set{u \in V(G) : \exists v \in U, \dist_G(v,u) \leq 1}$. 
For a vertex $v\in V(G)$ and a set of vertices $U\subseteq V(G)$, we say that $v$ \emph{is complete to} $U$ if $v$ is adjacent to every vertex of $U$.
Let $G$ be a graph and $S\subseteq E(G)$.
We denote by $G\big\vert_S$ the graph whose vertex set is $V(G)$ and whose edge set is $S$.
For any $U\subseteq V(G)$, we denote by $G[U]$ the subgraph of $G$ induced by $U$. % and for $G'=G[V']$ we denote by $G'\big\vert_S$ the graph $G\big\vert_S [V']$.
For any $U\subseteq V(G)$, we denote by $G-U$ the graph $G[V(G)\setminus U]$.
For any vertex $v\in V(G)$, we denote by $G-v$ the graph $G-\set{v}$.
%We denote by $G/S$ the graph whose vertices are the connected components of $G\big\vert_S$ and two components $A,B\subseteq V(G\big\vert_S)$ are adjacent in $G/S$ if and only if $\dist_G(A,B)=1$.
%This is equivalent to the regular notion of contracting the edges in $S$.
%If $x\in V(G)$, we denote by $x/S$ the component of $G\big\vert_S$ which contains $x$.
%This is equivalent to saying that $x/S$ is the vertex in $G/S$ which corresponds to $x$ after having contracted $S$.
%For a set $U\subseteq V(G)$ we denote by $U/S$ the set $\bigcup_{u\in U}u/S$.

Let $S \subseteq E(G)$. 
We denote by $G/S$ the graph whose vertices are in one-to-one correspondence to the connected components of $G\big\vert_S$ and two vertices $u,v\in V(G/S)$ are adjacent if and only if their corresponding connected components $A,B$ of $G\big\vert_S$ satisfy $\dist_G(V(A),V(B))=1$.
This is equivalent to the regular notion of contracting the edges in $S$. 
However, this definition allows us to make the notation in the proofs simpler and less confusing.

We say that a set $I \subseteq V(G)$ is \emph{independent} if the vertices contained in it are pairwise non-adjacent. 
We denote by $\alpha(G)$ the size of a maximum independent set in $G$. 
The decision problem \mis \ takes as input a graph $G$ and an integer $k$ and outputs \yes\, if and only if there is an independent set of size at least $k$ in $G$. 
We say that a set $U\subseteq V(G)$ is a \emph{clique} if every two vertices in $U$ are adjacent.
We denote by $\omega(G)$ the size of a maximum clique in $G$.
We call a set $U \subseteq V(G)$ a \emph{vertex cover}, if for every edge $uv \in E(G)$ we have that $u \in U$ or $v \in U$. 
The decision problem \mvc \ takes as input a graph $G$ and an integer $k$ and outputs \yes\, if and only if there is a vertex cover of size at most $k$ in $G$.  
We denote by $\tau(G)$ the size of a minimum vertex cover in $G$. 
Furthermore, we call a graph $M$ a \emph{matching} of a graph $G$, if $V(M) \subseteq V(G)$, $E(M) \subseteq E(G)$ and each vertex in $M$ has exactly one neighbour in $M$. 
We say that a matching is a \emph{maximum matching} if it contains the maximum possible number of edges and denote this number by $\mu(G)$.
Observe that we did not use the standard definition of a matching as a set of non-adjacent edges. 
This was done in order to simplify the notation in the proofs.
However, the edge set of a matching in our definition follows the conventional definition.

A graph without cycles is called a \emph{forest} and a connected forest is a \emph{tree}.
It is well-known that a tree has one more vertex than it has edges.
A graph is said to be \emph{chordal} if it has no induced cycle of length at least four.
A graph $G$ is \emph{bipartite} if we can find a partition of the vertices into two sets $V(G) = U \cup W$ such that $U$ and $W$ are both independent sets.
For a given graph $H$, we say that the graph $G$ is \emph{$H$-free} if it does not contain $H$ as an induced subgraph.

For a positive integer $i$ we denote by $P_i$ and $C_i$ the path and the cycle on $i$~vertices, respectively.
We call the graph which is given in Figure \ref{paw} a \emph{paw}.
\begin{figure}
\centering
\begin{tikzpicture}
\node[vertex](c1) at(0,0){};
\node[vertex](c2) at(30:1){};
\node[vertex](c3) at(180:1){};
\node[vertex](c4) at(330:1){};
\draw[hedge](c1)--(c4);
\draw[hedge](c1)--(c3);
\draw[hedge](c1)--(c2);
\draw[hedge](c2)--(c4);
\end{tikzpicture}
\caption{\label{paw}The paw}
\end{figure}

For a given graph parameter $\pi$ we say that a set $S\subseteq E(G)$ is \emph{$\pi$-contraction-critical} if $\pi(G/S)<\pi(G)$.
We say that a set $U\subseteq V(G)$ is \emph{$\pi$-deletion-critical} if $\pi(G-U)<\pi(G)$.

We will use the following two results. The first one is due to K\H{o}nig,  the second one is well-known and easy to see.
\begin{lemma}[see \cite{diestel}]
\label{Koenig}
Let $G$ be a bipartite graph. Then $\mu(G) = \tau(G)$. 
\end{lemma}
\begin{lemma}
\label{mvc+indset}
Let $G$ be a graph and let $I\subseteq V(G)$ be a maximum independent set. Then $V(G) \setminus I$ is a minimum vertex cover and hence $\tau(G) +\alpha(G)= |V(G)|$. 
\end{lemma}
In \cite{poljak} it was shown that \mis \ is NP-complete in $C_3$-free graphs. 
This and \Cref{mvc+indset} imply the following corollary.

\begin{corollary}
\label{mvcNPcomp}
\mvc \ is NP-complete in $C_3$-free graphs.
\end{corollary}

\section{Hardness proofs}\label{hardness}

We begin by restating \textsc{Vertex Cover} as a satisfiability problem in order to simplify the notation in the proofs.

\begin{center}
\fbox{
\begin{minipage}{4.5in}
\wps
\begin{description}
\item[Instance:] A variable set $X$,  a clause set $C$ in which all clauses contain exactly two literals and every literal is positive, as well as an integer $k$.
\item[Question:] Is there a truth assignment of the variables (that is, a mapping $f\colon\, X\rightarrow \set{\textrm{true},\,\textrm{false}}$) such that at least one literal in each clause is true and there  are at most $k$ variables which are true.
\end{description}
\end{minipage}}
\end{center}

If $\Phi=(G,k)$ is an instance of \textsc{Vertex Cover} then taking $X=V(G)$ as the variable set and $C=\set{(u \lor w)\colon\, uw\in E(G)}$ as the set of clauses yields an instance $(X,C,k)$ of \wps\ which is clearly equivalent to $\Phi$. 
Since \textsc{Vertex Cover} is known to be NP-hard (see \Cref{mvcNPcomp}), it follows that \wps\ is NP-hard, too.

Let $G$ be a graph and $S, S' \subseteq E(G)$ such that for every connected component $A$ of $G\big\vert_S$ there is a connected component $A'$ of $G\big\vert_{S'}$ with $V(A)=V(A')$. 
Then, $G/S = G/S'$ and thus we get the following corollary.

%If $G$ is a graph and $S\subseteq E(G)$, consider the set $S'\subseteq S$ formed by the edges of any spanning forest of $G\big\vert_S$. 
%Since the  connected components of $G\big\vert_S$ and $G\big\vert_{S'}$ are the same, we obtain the following corollary.

\begin{corollary}\label{forest}
Let $G$ be a graph and $S\subseteq E(G)$ a minimal $\alpha$-contraction-critical set of edges. 
Then, $G\big\vert_S$ is a forest.
\end{corollary}

\begin{theorem}
\label{contchordal}
$1$-\contracindep{} is \np-complete in chordal graphs.
\end{theorem}
\begin{proof}
It was shown in \cite{gavril} that \mis \ can be solved in polynomial time for chordal graphs. 
Since the family of chordal graphs is closed under edge contractions, for a given chordal graph $G$ and a set $S\subseteq E(G)$, it is possible to check in polynomial time whether $S$ is $\alpha$-contraction-critical. 
It follows that 1-\contracindep{} is in \np\ for chordal graphs. 
In order to show \mbox{\np-hardness}, we reduce from \wps, which was shown to be \np-hard above. 
Let $\Phi=(X,C,k)$ be an instance of \wps. 
We construct a chordal graph $G$ such that $(G,k)$ is a \yes-instance for 1-\contracindep{} if and only if $\Phi$ is a \yes-instance for \wps, as follows:

For every variable $x\in X$, we introduce a set of vertices $G_x$ with $G_x=\set{v_x}\cup K_x$, where $K_x$ is a set of $2k+1$ vertices which induce a clique. 
We make $v_x$ complete to $K_x$. 
For every clause $c\in C$, we introduce a vertex $v_c$. We define $K_C = \bigcup_{c \in C} \set{v_c}$. We add edges so that $G[K_C]$ is a clique. For every clause $c\in C$, $c= (x\lor y)$, we make $v_c$ complete to $K_{x}$ and $K_{y}$ (see Figure 1 for an example).

Observe first that the graph $G$ is indeed chordal: if a cycle of length at least four contains at least three vertices of $K_C$, it follows immediately that the cycle cannot be induced, since $K_C$ induces a clique.
Otherwise, such a cycle contains at most two vertices of $K_C$.
If there are two vertices $w$ and $w'$ of the cycle which are contained in $G_x$ and $G_y$, respectively, with $x,y \in X, x\neq y$, then the cycle has to contain a chord in $G[K_C]$ and is thus not induced. 
If all vertices of the cycle are in $K_C\cup G_x$ for some fixed $x\in X$, then there are at least two vertices $w$ and $w'$ contained in $K_x$.
Hence, the cycle cannot be induced since $w$ and $w'$ are adjacent and have the same neighbourhood.
It follows that $G$ cannot have any induced cycle of length at least 4 and is thus chordal.

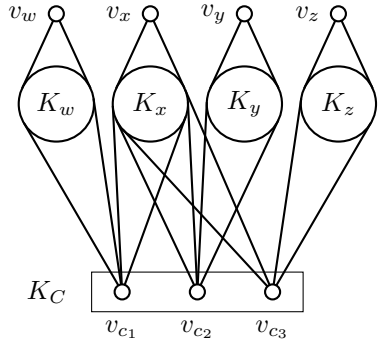
\begin{figure}
\centering
\begin{tikzpicture}
\def\radius{0.5}
\def\k{1}
\def\h{2.5}

%Gadgets
\begin{scope}[shift = {(-1.875*\k,\h)}]
	\gadget{w}{\radius}
\end{scope}

\begin{scope}[shift = {(-0.625*\k,\h)}]
	\gadget{x}{\radius}
\end{scope}

\begin{scope}[shift = {(0.625*\k,\h)}]
	\gadget{y}{\radius}
\end{scope}

\begin{scope}[shift = {(1.875*\k,\h)}]
	\gadget{z}{\radius}
\end{scope}

%Tangenten zu Gadgets

\tangent{-1.875*\k}{\h}{\radius}{-1*\k}{0}

\tangent{-0.625*\k}{\h}{\radius}{0}{0}
\tangent{-0.625*\k}{\h}{\radius}{-1*\k}{0}
\tangent{-0.625*\k}{\h}{\radius}{1*\k}{0}

\tangent{0.625*\k}{\h}{\radius}{0}{0}

\tangent{1.875*\k}{\h}{\radius}{1*\k}{0}

%Clause vertices
\node[vertex, ] (c3) at (-1*\k,0){};
\node[vertex,] (c4) at (0*\k,0){};
\node[vertex,] (c5) at (1*\k,0){};
  
\node[] (c3t) at (-1*\k,-0.5){$v_{c_1}$};
\node[] (c4t) at (0*\k,-0.5){$v_{c_2}$};
\node[] (c5t) at (1*\k,-0.5){$v_{c_3}$};
  
 %Box around clause vertices
  
\node[] (Kct) at (-2*\k,0){$K_C$};
\node[rectangle,draw, minimum width = 80*\k , minimum height = 15*\k](r) at (0*\k,0.0){};
  
\end{tikzpicture}

\caption{This is the graph corresponding to the instance of \wps \ given by the variables $w,x,y,z$ and the clauses $c_1 = w\lor x, c_2 = x\lor y$ and $c_3 = x\lor z$. The rectangular box corresponds to $G[K_C]$, the vertices contained in it induce a clique. Every set $K_i$ induces a clique and the lines between a vertex and a set $K_i$ mean that this vertex is complete to $K_i$.}
\end{figure}

Since $G_x$ induces a clique for every $x\in X$, it can contain at most one vertex in any independent set; the same applies to $K_C$. 
Thus, $\alpha(G)\leq \vert X\vert +1$. 
Let $c\in C$. 
Since the set $\{v_x \colon x\in X\}\cup \set{v_c}$ is an independent set of size $\vert X\vert +1$, it follows that $\alpha(G)=\vert X\vert +1$.

Let us assume that $\Phi$ is a \yes-instance of \wps.
Let $X_p$ be the set of positive variables of a satisfying assignment of $\Phi$.
For each $x \in X_p$, let $e_x$ be an edge incident to $v_{x}$ and let $S=\set{e_x\vert x\in X_p}$.
Let $G'=G/S$. We claim that $\alpha(G')<\alpha(G)$.
To see this, observe first that for any $x\in X_p$, contracting $e_x$ is equivalent to deleting the vertex $v_{x}$, since $N_G(v_{x})=K_{x}$ induces a clique.
Therefore, we have that $G'\simeq G-\set{v_x:\, x\in X_p}$.
Suppose for a contradiction that there is an independent set $I$ of $G'$ of size $|X|+1$.
Since $\vert I\cap K_x\vert\leq 1$ (for $x\in X_p)$ and $\vert I\cap G_x\vert\leq 1$ (for all $x\in X\setminus X_p$), it follows that there exists $c\in C$ such that $v_c\in K_C\cap I$.
Furthermore, the inequalities above all have to be equalities.
By the choice of $X_p$, it follows that there is $x\in X_p$ such that $x$ is a literal in~$c$.
Since $\vert I\cap K_x\vert=1$, there is a vertex $w\in I\cap K_x$ which is adjacent to $v_c$, contradicting the fact that $I$ is independent.
It follows that $S$ is $\alpha$-contraction-critical.

For the other direction, assume that $\Phi' = (G,k)$ is a \yes-instance of 1-\contracindep{}. Let $S$ be a minimum $\alpha$-contraction-critical set of edges such that $\vert S\vert\leq k$. 
By \Cref{forest}, the graph $G\big\vert_S$ is a forest.

For any $x\in X$, there is a vertex $u_x\in K_x\setminus V(S)$. 
This follows from the\break fact that $k$ edges can be incident to at most $2k$ vertices and\break $\vert K_x\vert =2k+1$.  
Let $H$ be the graph with vertex set $V(H) = K_C$ and edge set $E(H) = \set{uv\in S:u,v \in K_C}$.

Suppose for a contradiction that there is a connected component $T$ of $H$ such that for every $x\in X$ with $\dist_G(G_x,V(T))=1$ we have $G_x\cap V(S)=\varnothing$.
In other words, for every $c=(x\lor y)\in C$ with $v_c\in V(T)$ we have $G_x\cap V(S)=G_y\cap V(S)=\varnothing$.
So we have that $N_G[V(T)]\cap V(S)\subseteq V(T)$, and thus $T$ is also a connected component in $G\big\vert_S$.
For every $x\in X$ the set $\set{u_x}$ is a connected component in $G\big\vert_S$, that is, $u_x$ is not incident to any edge in $S$.
Further, for every $x\in X$ where $\dist_G(G_x,V(T))=1$, we have that $G_x\cap V(S)=\varnothing$ and thus $\set{v_x}$ is a connected component in $G\big\vert_S$.
Let $X_1= \set{x\in X\colon\, \dist_G(u_x,V(T))=1}$ and $X_2=X\setminus X_1$.
We claim that the set $I=T \cup\set{\set{v_x}\colon\, x\in X_1}\cup\set{\set{u_x}\colon\, x\in X_2}$ is a set of connected components of $G\big\vert_S$ which correspond to vertices in $G/S$ who are pairwise at distance at least two.
In other words, $I$ corresponds to an independent set in $G/S$.
The connected components of the set $I$ are pairwise at distance at least two in $G$ and so they correspond to pairwise non-adjacent vertices in $G/S$.
Thus, the set $I$ corresponds indeed to an independent set of cardinality $\vert X\vert+1$, a contradiction to the assumption that $S$ is $\alpha$-contraction-critical.
It follows that there is no connected component $T$ of $H$ such that for every $x\in X$ with $\dist_G(G_x,V(T))=1$ we have $G_x\cap V(S)=\varnothing$.

We can obtain a truth assignment of the variables satisfying $\Phi$ as follows: Set every $x$ to true for which $G_x\cap V(S)$ is non-empty.
For every clause $c=(x\lor y)\in C$ for which both $G_x\cap V(S)$ and $G_y\cap V(S)$ are empty, set one of its variables to true. 
This assignment is clearly satisfying, it remains to show that we set at most $\vert S\vert\leq k$ variables to true. 
Consider a connected component $T$ of $H$. 
Recall that $T$ is a tree, and so its number of vertices is one more than its number of edges. 
We have shown that there is a vertex $v_c\in V(T)$, $c=(x\lor y)$, for which $G_x\cap V(S)\neq\varnothing$.
Thus, there are at most $\vert E(T)\vert$ vertices $v_c\in T$, $c=(x\lor y)$, for which both $G_x\cap V(S)$ and $G_y\cap V(S)$ are empty.
This implies that for every connected component $T$ of $H$ we set at most $\vert E(T)\vert$ variables to true.
Further, the number of variables $x\in X$ which we set to true because $G_x\cap V(S)\neq\varnothing$ is at most the number of edges of $S$ which are not contained in $G[K_C]$.
This shows that, in total, we set at most $\vert S\vert$ variables to true, which concludes the proof.
\end{proof}

Interestingly, \delindep{} and 1-\contracindep{}\break are equivalent on the instance $\Phi'$ constructed in the proof of \Cref{contchordal} and\break thus the same construction can be used to show \np-hardness of \delindep{} in chordal graphs.

\begin{theorem}
\label{delchordal}
\delindep{} is \np-complete in chordal graphs.
\end{theorem}
\begin{proof}
It has been shown in \cite{gavril} that it is possible to determine the independence number of chordal graphs in polynomial time. 
Since chordal graphs are closed under vertex deletion, it is possible to check in polynomial time whether the deletion of a given set of vertices reduces the independence number. 
Hence \mbox{\delindep{}} is in \np \ for chordal graphs.

In order to show \np-hardness, we reduce from \wps. Let $\Phi$ be an instance of \wps,  $\Phi=(X, C,k)$. 
Let\break $\Phi'=(G,k)$ be the instance of 1-\contracindep{} which is described in \Cref{contchordal} and which has been shown to be equivalent to $\Phi$. 
Further, let $K_x, G_x$ and $v_x$ for each $x\in X$, $K_C$, and $v_c$ for each $c\in C$ be as in the proof of \Cref{contchordal}. 
Recall that we have shown that $\alpha(G)=\vert X\vert +1$ and that $G$ is chordal.

We show that $\Phi'$ is a \yes-instance of \delindep{} if and only if $\Phi$ is a \yes-instance of \wps.

Assume first that $\Phi$ is a \yes-instance of \wps \ and that $X_p$ is the set of positive variables in a satisfying assignment of $\Phi$. We have shown in the proof of \Cref{contchordal} that $\alpha(G- \set{v_x\colon x\in X_p})<\alpha(G)$, hence $(G,k)$ is a \yes-instance of \delindep{}.

Conversely, assume that $\Phi'$ is a \yes-instance of \delindep{} and let $W$ be an $\alpha$-deletion-critical set of vertices of cardinality at most $k$. 
For every $x\in X$ there is $u_x\in K_x\setminus W$, since $\vert W\vert <\vert K_x\vert$.
%We may assume that $S$ is minimal. Since $S$ is minimal we know that we don't delete vertices in the cliques $K_x, x\in X$.
Define a set\break $Z=\set{x\in X\colon v_x\in W}$ and initialize a set $Z'=\varnothing$. 
For every clause $c\in C$ with $v_c \in W$ we choose one of the variables contained in $c$ and add it to $Z'$. 
We claim that setting the variables of $Z\cup Z'$ to true yields a satisfying assignment of $\Phi$. 
Observe first that $\vert Z\cup Z'\vert\leq\vert W\vert\leq k$ by construction. 
Suppose for a contradiction that there is a clause $c\in C$, $c=(x\lor y)$, such that neither $x$ nor $y$ is contained in $Z\cup Z'$. 
It follows that $v_{x},v_{y},v_c\notin W$. 
But then $\{v_c,v_{x},v_{y}\} \cup \set{u_z\colon\,z\in X\setminus \set{x,y}}$ is an independent set of size $|X|+1$ in $G-W$, a contradiction to the $\alpha$-deletion-criticalness of $W$. 
Hence the assignment is satisfying and the theorem follows.
\end{proof}

\begin{corollary}
\label{delperfect}
\delindep{} is \np-complete in perfect graphs.
\end{corollary}

The last theorem in this section answers a question asked in \cite{CliqueLatest}. 
Indeed, \break \Cref{C3P1Hard} settles the missing case of \cite[Theorem 24]{CliqueLatest} and completes the complexity dichotomy for $H$-free graphs, which is as follows.

\begin{theorem}
Let $H$ be a graph. 
If $H$ is an induced subgraph of $P_4$ or of the paw, then \textsc{Contraction Blocker}$(\omega)$ is polynomial-time solvable for $H$-free graphs, otherwise it is NP-hard or co-NP-hard for $H$-free graphs.
\end{theorem}

\begin{theorem}\label{C3P1Hard}
The decision problem \contracclique \ is NP-hard in $(C_3 + P_1)$-free graphs.
\end{theorem}
\begin{proof}
We use a reduction from \mvc \ in $C_3$-free graphs which is NP-complete due to \Cref{mvcNPcomp}. 
Let $(G,k)$ be an instance of \mvc \ where $G$ is a $C_3$-free graph. 
Since \mvc\ is trivial to solve on a graph without edges, we can assume that $E(G)$ is non-empty. 
We construct an instance $(G',k)$ of \contracclique \  such that $(G,k)$ is a \yes-instance of \mvc \ if and only if $(G',k)$ is a \yes-instance of \contracclique\ and $G'$ is $(C_3+P_1)$-free. 
Let $G'$ be a graph with $V(G') = V(G) \cup \set{w}$, $w\notin V(G)$, and $E(G') = E(G) \cup \{wv, v\in V(G)\}$. 
In other words, we add a universal vertex $w$ to $G$ in order to obtain $G'$.

Since $G$ is $C_3$-free, every copy of $C_3$ in $G'$ has to contain $w$. 
Furthermore, since $w$ is adjacent to every other vertex in $V(G')$, it follows that every vertex of $G'$ has distance at most one to every copy of $C_3$. 
Thus, $G'$ is $(C_3 + P_1)$-free. 
Also, note that $\omega(G') = 3$ and that every maximum clique in $G'$ is a copy of $C_3$ which contains $w$ and exactly two vertices of $V(G)$.

Let us assume that $(G,k)$ is a \yes-instance of \mvc. 
Let \break $\set{v_1,\dots, v_k}\subseteq V(G)$ be a vertex cover of $G$. 
Set $S = \set{v_iw\colon\,i \in \set{1,\dots, k}} $ and let $G^* = G'/S$. 
We claim that $S$ is $\omega$-contraction-critical. 
Notice that the contraction of an edge $vw \in S$ is equivalent to deleting the vertex $v$, since the new vertex remains adjacent to all other vertices. 
Thus, $G^*$ is isomorphic to $G-(V(S)\setminus\set{w})$. 
Since $\set{v_1,\dots, v_k}$ is a minimum vertex cover of $G$, there are no edges in $G^*-w$, meaning that $G^*$ is $C_3$-free and thus $\omega(G^*) \leq 2$. 
Hence $(G',k)$ is a \yes-instance of \contracclique.

For the other direction, assume that $(G',k)$ is a \yes-instance of \contracclique.
Let $S\subseteq E(G')$ be a minimum $\omega$-contraction-critical set of edges with $\vert S\vert\leq k$ and let $G^* = G'/S$.

We construct a set $U$ of vertices of $G$ as follows: For the connected component $T$ of $G'\big\vert_S$ that contains $w$, add every vertex of $V(T)$ except $w$ to $U$. 
For every other connected component $T$ of $G'\big\vert_S$ we add to $U$ all vertices of $V(T)$ except one, which can be chosen arbitrarily.
We claim that $U$ is a vertex cover of $G$ of size at most $k$.

To see that $\vert U\vert\leq k$, observe that for every connected component $T$ of $G'\big\vert_S$ we have added $\vert V(T)\vert -1$ vertices to $U$. 
Since $T$ is a tree (see \Cref{forest}), we have that $\vert V(T)\vert -1=\vert E(T)\vert$. 
Thus, we have added as many vertices to $U$ as there are edges in $S$ and hence $\vert U\vert =\vert S\vert\leq k$.

In order to show that $U$ is a vertex cover, suppose for a contradiction that there is an edge $uv\in E(G)$ for which neither $u$ nor $v$ is contained in $U$. 
Consider the connected components $A_u$, $A_v$ and $A_w$ of $G'\big\vert_S$ which contain $u$, $v$ and $w$, respectively. 
It follows from the construction of $U$ that in every connected component $T$ of $G'\big\vert_S$ there is at most one vertex of $T$ which is not contained in $U$. 
Hence, $A_u\neq A_v$. 
We have that $w\not\in U$ by construction, so the same argument can be used to show that $A_u\neq A_w$ and $A_v\neq A_w$. 
Thus, $A_u$, $A_v$, $A_w$ correspond to three different vertices in $G^*$ and since the components are pairwise at distance one, their corresponding vertices induce a $C_3$ in $G^*$, a contradiction to $S$ being $\omega$-contraction-critical.
Thus, $U$ is a vertex cover in $G$ and $(G,k)$ a \yes-instance of \mvc.
\end{proof}

\section{Algorithms}\label{algorithms}

In this section we give a polynomial-time algorithm for $d$-\textsc{Contraction\break Blocker}($\alpha$) in bipartite graphs.

\begin{theorem}
\label{maxnumedgecontraction}
Let $G$ be a connected, bipartite graph with $|V(G)| \geq 2d+2$ and $\alpha(G) \geq d+1$, where $d\geq 1$ is an integer. Then $(G,2d+1)$ is a \yes-instance of \dcontracindep.
\end{theorem}
\begin{proof}
Let $G$ be a bipartite graph with $|V(G)| \geq 2d+2$ and $\alpha(G) \geq d+1$. 
Let $M$ be a maximum matching of $G$. 
Since $G$ is connected, $M$ is non-empty. 
Consider the following algorithm which constructs a tree $T$, which is a subgraph of $G$.

\begin{algorithm}
\caption{}
\label{Algobuildtree}
\begin{algorithmic}[2]
\Require A bipartite graph $G$, a maximum matching $M$ in $G$, an integer $d\geq1$
\Ensure A tree $T$
\State Choose an arbitrary edge $uu'\in E(M)$.
\State Set  $V(T) = \{u,u'\}, E(T) = \{uu'\}$.
\While {$|E(T)| \leq2d-1$}
\State Choose two vertices $w\in N_G(T)\setminus V(T)$, and $w'\in N_G(w)\cap V(T)$. \label{algoline:choosev}
\If {$w\in V(M)$} 	
\State Let $v \in V(M)$ s.t. $ vw \in E(M)$.
\State $V(T) = V(T) \cup \{v,w\}$, $E(T) = E(T) \cup \{w'w,vw\}$ \label{algoline:chooseu}
\Else \ $V(T) = V(T) \cup \{w\}$, $E(T) = E(T) \cup \{w'w\}$
\EndIf
\EndWhile
\State \Return $T$
\end{algorithmic}
\end{algorithm}

We claim that the resulting graph $T$ is a tree. 
Indeed, the initial graph is a single edge and thus a tree. 
Further, observe that every time there are vertices and edges added to $T$ in lines 7 or 8, the resulting graph remains connected and the number of added vertices and added edges is the same. 
It follows that $T$ is connected and has exactly one more vertex than it has edges and is thus a tree.
It is easy to see that $T$ has $2d$ or $2d+1$ edges.

We consider the graph $G' = G-V(T)$.
For every $v \in V(M)\cap V(T)$ the unique vertex $u\in V(M)$ with $uv\in E(M)$ is also contained in $V(T)$ and $uv\in E(T)$.
Thus, there are at most $\Big \lfloor \frac{|V(T)|}{2}\Big \rfloor$ edges in $E(M)$ which have an endvertex in $T$.
Since $M - V(T)$ is a matching in $G'$ we have that $\mu(G') \geq \mu(G)-\Big \lfloor \frac{|V(T)|}{2}\Big \rfloor$.
Applying \Cref{Koenig} and \Cref{mvc+indset}, we get for the independence number of $G'$:
\begin{align*}
\alpha(G') = |V(G')| - \mu(G') &\leq |V(G)|-|V(T)|- \mu(G)+\bigg \lfloor \frac{|V(T)|}{2}\bigg \rfloor \\
&= \alpha(G) - \bigg \lceil \frac{|V(T)|}{2}\bigg \rceil = \alpha(G)-d-1.
\end{align*}
Let $G^*=G/E(T)$. 
Observe that $G\big\vert_{E(T)}$ contains exactly one connected component, say $A$, which has more than one vertex, namely the connected component corresponding to $T$.
Let $v^*\in V(G^*)$ be the vertex which corresponds to $A$.
Since $G^*-v^*$ is isomorphic to $G'$, we obtain that $\alpha(G^*) \leq \alpha(G')+1 \leq \alpha(G)-d$.
\end{proof}

\begin{algorithm}
\caption{}
\label{Algobip}
\begin{algorithmic}[2]

\Require A bipartite graph $G$, an integer $k$, a fixed integer $d$
\Ensure \textsc{Yes} if $(G,k)$ is a \yes-instance of \dcontracindep, \textsc{No} if not
\For {every $S \subseteq E(G)$ of size at most $k$}
	\State Let $\beta = 0$.
	\State Let $G'= G/S$.
	\State Let $ U=\set{v\in V(G'): v\textit{ corresponds to a connected component of }G\big\vert_S \right.$\par\hspace{150pt}  $\left.\textit{ which contains at least 2 vertices} }. $
	\For {every subset $U' \subseteq U$}
		\If {$U'$ is independent }
			\State $\beta = \max(\beta, \alpha(G'-(U\cup N_{G'}(U'))) + |U'|)$ 
		\EndIf
	\EndFor
	\If{$\beta \leq \alpha(G)-d$} 
	\State \Return \textsc{Yes} 
	\EndIf
\EndFor
\State \Return \textsc{No}
\end{algorithmic}
\end{algorithm}

\begin{theorem}
\dcontracindep \ is solvable in polynomial time in bipartite graphs.
\end{theorem}
\begin{proof}
Let $G$ be a bipartite graph and $k$ a positive integer. 
If $|V(G)| \leq 2d+1$ there are at most $2^{d(d+1)}$ subsets of $E(G)$ and at most $2^{2d+1}$ subsets of $V(G)$. 
We can check for every subset $S\subseteq E(G)$ if $\alpha(G/S)\leq\alpha(G)-d$ in constant time by computing the graph $G/S$ and checking for each subset of $V(G/S)$ if it is independent. 
Thus, we can check in constant time if $G$ is a \yes-instance for \dcontracindep.

Since contracting edges in a non-empty graph cannot reduce the number of vertices to zero, it follows that if $\alpha(G) \leq d$ it is not possible to reduce $\alpha(G)$ by $d$ via edge-contractions. 
Hence, we can assume that $|V(G)| \geq 2d+2$ and $\alpha(G) \geq d+1$. 
By \Cref{maxnumedgecontraction}, we know that for $k\geq 2d+1$, it is always possible to contract at most $k$ edges to reduce the independence number of $G$ by at least d, so we can further assume that $k\leq 2d$. 

%A set $I_D$ is a maximum independent set containing $D \subseteq V(G)$ if and only if $D$ is independent, $N_G(D) \cap I_D = \varnothing$ and $I_D \cap V(G)\setminus N_G[D]$ is a maximum independent set of $G- N_G[D]$. For the size of $I_D$ it holds $|I_D| = \alpha(G-N_G[D]) + |D|$. 

Consider now \Cref{Algobip} which takes as input $G,k$ and $d$ and outputs $\yes{}$ or $\no{}$.
\Cref{Algobip} considers every subset $S\subseteq E(G)$ of edges of cardinality at most $k$ and computes $\alpha(G/S)$. 
If there is some $S$ such that $\alpha(G/S)\leq\alpha(G)-d$ then we return $\yes$, and $\no$ otherwise. 
In order to compute $\alpha(G/S)$ for such a subset $S$ of edges, we first set $G'=G/S$ and consider the set of vertices $U\subseteq V(G')$ which have been formed by contracting some edges in $S$ (see line 4 of the algorithm). 
Observe that $G[V(G')\setminus U]$ is isomorphic to $G-V(S)$ and induces thus a bipartite graph. 
Every independent set of $G'$ can be partitioned into a set $U'\subseteq U$ and a set $W\subseteq V(G')\setminus (U\cup N_{G'}(U'))$. 
Thus, we can find the independence number of $G'$ by considering every independent subset $U'$ of $U$ and computing $\alpha(G'- (U\cup N_{G'}(U')))+\vert U'\vert$. 
The largest of these values is then $\alpha(G')$. 
The independence number of the bipartite graph $G'-(U\cup N_{G'}(U'))$ can be computed in polynomial time, see \Cref{mvc+indset} and \cite{ahuja}.

The number of subsets of $E(G)$ of cardinality at most $k$ is in $O(\vert E(G)\vert^k)=O(\vert V(G)\vert^{4d})$.
For any such subset $S$, the number of subsets $U'\subseteq U$ is at most $2^k\leq 2^{2d}$.
Thus, the running time of \Cref{Algobip} is polynomial.
\end{proof}
\bibliographystyle{splncs04}
\bibliography{ref}

\end{document}